\documentclass[twoside,11pt]{amsart}
\usepackage{indentfirst,latexsym,bm}
\usepackage{amsfonts}
\usepackage{amssymb}
\usepackage{amsmath}
\usepackage{amsbsy}
\usepackage{dsfont}
\usepackage{amsthm}
\usepackage{hyperref}
\usepackage{amscd}
\usepackage[all]{xy}
\usepackage{chemarrow}
\usepackage{multirow}
\allowdisplaybreaks[4]
\usepackage[titletoc]{appendix}

\newcommand{\K}{\mathds{k}}

\newcommand{\G}{\mathbf{G}}

\newcommand{\Pp}{\mathcal{P}}

\newcommand{\fA}{\mathfrak{A}}

\newcommand\Aut{\operatorname{Aut}}

\newcommand\Ker{\operatorname{Ker}}

\newcommand\id{\operatorname{id}}

\newcommand{\Hom}{\operatorname{Hom}}
\newcommand\ord{\operatorname{ord}}

\begin{document}
	
	\title[Exact Factorizations of Rank-One Hopf Algebras I]{Exact Factorizations of Rank-One Pointed Hopf Algebras in Positive Characteristic, I}
	\author{Rongchuan Xiong}
	\address{Department of Mathematics, Changzhou University, Changzhou 213164, China}
	\email{rcxiong@foxmail.com}
	
	\makeatletter
	\@namedef{subjclassname@2020}{\textup{2020} Mathematics Subject Classification}
	\makeatother
	
	\date{}
	
	\begin{abstract}
		We classify exact factorizations of the third-type rank-one pointed Hopf algebras over an algebraically closed field of positive characteristic.
		
		The main step is a complete classification of matched pairs between such an algebra and a group algebra. It turns out that the group-like part must form a matched pair of finite groups, while the only possible action on the skew-primitive generator is a shift by a scalar multiple of $1-g$, where $g$ is the distinguished group-like element, encoded by a single group homomorphism. The bicrossed product is again a third-type rank-one pointed Hopf algebra, and we give a necessary and sufficient condition for two such products to be isomorphic as Hopf algebras.
		
		Consequently, up to interchanging the two factors, exact factorizations of a fixed third-type algebra correspond bijectively to exact factorizations of its underlying finite group, with the distinguished group-like element lying in the rank-one factor.
		
		As an application, all matched pairs between the Radford algebra and cyclic group algebras are classified, and the corresponding exact factorizations are determined explicitly when the cyclic group has prime order.

	\end{abstract}

	\maketitle
	
	\newtheorem{question}{Question}
	\newtheorem{defi}{Definition}[section]
	\newtheorem{conj}{Conjecture}
	\newtheorem{thm}[defi]{Theorem}
	\newtheorem{lem}[defi]{Lemma}
	\newtheorem{pro}[defi]{Proposition}
	\newtheorem{cor}[defi]{Corollary}
	\newtheorem{rmk}[defi]{Remark}
	\newtheorem{example}{Example}[section]

	\section{Introduction}
	A Hopf algebra $E$ is said to admit an \emph{exact factorization} through Hopf subalgebras $B$ and $T$ if the multiplication map
	\[
	B\otimes T\longrightarrow E,\qquad b\otimes t\longmapsto bt,
	\]
	is bijective. When $B$ and $T$ are regarded as fixed Hopf subalgebras of $E$, we write $E=BT$. By Majid's factorization theorem \cite[Proposition~3.12]{Majid90} (see also\cite[Theorem~7.2.3]{Majid95}), $E$ factors through $B$ and $T$ if and only if $E$ is isomorphic to a bicrossed product associated with a matched pair of $B$ and $T$. Thus, the factorization problem reduces to classifying the relevant matched pairs and then determining whether the corresponding bicrossed products are isomorphic to the Hopf algebra under consideration.
	
	Bicrossed products have been studied in many concrete settings. Agore, Bontea, and Militaru \cite{ABM14} developed a systematic approach to their classification. Subsequent work has treated, among other examples, bicrossed products of two Sweedler Hopf algebras \cite{Bontea14}, two Taft algebras \cite{Agore18-Taft}, Taft algebras with group algebras \cite{Agore18-Cn,Agore-Nastasescu17}, $H_8$ with $H_4$ \cite{Lu-Ning-Wang20}, and generalized Taft algebras with group algebras \cite{Wang-Cheng-Lu22}.
	
    In positive characteristic, Scherotzke \cite{Scherotzke} classified finite-dimensional pointed rank-one Hopf algebras generated as algebras by the first term of the coradical filtration. The classification consists of three types. The third type has no counterpart in characteristic zero \cite{KropRadford}. In this type, the conjugation action of the distinguished group-like element on the skew-primitive generator is not semisimple; it is of additive type. This paper is the first of a series devoted to exact factorizations of rank-one pointed Hopf algebras in positive characteristic. It treats this third type. The remaining two types have semisimple conjugation action on the distinguished skew-primitive element. They will be addressed separately.

	Over an algebraically closed field, this third type can be normalized; the resulting Hopf algebra is denoted by $\mathfrak A^3_G(g,f)$, where $G$ is a finite group, $g\in Z(G)$ is the distinguished group-like element, and $f:G\to(\K,+)$ is a group homomorphism with $f(g)=1$, subject to the compatibility condition imposed by the defining relation $x^p=x$; see Definition~\ref{def:fA3}. The basic example is the Radford algebra \cite{Radford77}, which is the unique noncommutative noncocommutative pointed Hopf algebra of dimension $p^2$ over $\K$ \cite{WW14}.

	In this paper we determine all exact factorizations of the normalized third-type algebras $\mathfrak A^3_G(g,f)$. The main technical step is a complete classification of matched pairs between a group algebra and a third-type rank-one pointed Hopf algebra. Let
	\[
	A=\K[\Gamma],\qquad H=\mathfrak A^3_G(g,f),
	\]
	where $\Gamma$ is finite. For every matched pair $(A,H,\triangleright,\triangleleft)$, the distinguished generators satisfy the rigidity relations: for $h\in\Gamma$, 
	\[
	g\triangleright h=h,\qquad x\triangleright h=0,\qquad g\triangleleft h=g.
	\]
	The right action on $x$ is necessarily
	\[
	x\triangleleft h=x+\gamma_h(1-g)
	\]
	for a uniquely determined additive parameter $\gamma_h$. The classification is summarized as follows.
	
	\begin{thm}[Matched-pair classification]\label{thm:intro-matchedpair}
		Let $A=\K[\Gamma]$ and $H=\mathfrak A^3_G(g,f)$. Matched pairs $(A,H,\triangleright,\triangleleft)$ are in bijection with the following data:
		\begin{enumerate}
			\item a matched pair of finite groups $(\Gamma,G,\triangleright,\triangleleft)$ such that
			\[
			g\triangleright h=h,\qquad g\triangleleft h=g,\qquad \forall h\in\Gamma;
			\]
			\item a group homomorphism $\gamma:\Gamma\to \K_g^+$, where
			\[
			\K_g=\{a\in\K\mid (a^p-a)(1-g^p)=0\},
			\]
			satisfying
			\[
			f(t\triangleleft h)=f(t)+\gamma_{t\triangleright h}-\gamma_h,
			\qquad \forall t\in G,\ \forall h\in\Gamma.
			\]
		\end{enumerate}
		For these data,
		\[
		\K[\Gamma]\bowtie\mathfrak A^3_G(g,f)
		\cong
		\mathfrak A^3_{\Gamma\bowtie G}(g,F),
		\qquad
		F(h,t)=f(t)-\gamma_h.
		\]
	\end{thm}
	
	In particular, every bicrossed product appearing in Theorem~\ref{thm:intro-matchedpair} again belongs to the normalized third-type family. Since many previously studied examples have \(G=\langle g\rangle\), the left action is automatically trivial there; for general \(G\), however, it need not be (see Remark~\ref{rmk:nontrivial-left-action}).
	
	We also obtain an isomorphism criterion.
	
	\begin{thm}[Isomorphism classification]\label{thm:intro-isom}
		Let $(L,F)$ and $(L',F')$ be the group/function data attached to two bicrossed products in Theorem~\ref{thm:intro-matchedpair}. Then
		\[
		\mathfrak A^3_L(g,F)\cong\mathfrak A^3_{L'}(g,F')
		\]
		as Hopf algebras if and only if there is a group isomorphism $\Theta:L\to L'$ such that
		\[
		\Theta(g)=g,\qquad F'\circ\Theta=F.
		\]
		For matched pairs with trivial left action, the factor-preserving version reduces to the conjugacy condition in Theorem~\ref{thm:isomorphism} below.
	\end{thm}
	
	This classification also determines all exact factorizations of a fixed third-type algebra.
	
	\begin{thm}[Exact factorization structure]\label{thm:intro-factorization}
		Let $H=\mathfrak A^3_G(g,f)$. Up to interchanging the two factors, exact factorizations $H=BT$ by proper Hopf subalgebras are in bijection with nontrivial exact factorizations
		\[
		G=G_B\Gamma,
		\qquad G_B\cap\Gamma=\{1\},
		\qquad g\in G_B,
		\qquad \Gamma\ne\{1\}
		\]
		of the underlying group. The corresponding Hopf subalgebras are
		\[
		B=\mathfrak A^3_{G_B}\bigl(g,f|_{G_B}\bigr),
		\qquad
		T=\K[\Gamma].
		\]
	\end{thm}
	
	Thus the rank-one factor is obtained by restricting the datum to a subgroup
	containing $g$, while the other factor is a group algebra. The canonical
	matched pair and the classification of exact factorizations up to Hopf
	automorphism are given in Section~\ref{sec:internal} and
	Corollary~\ref{cor:factorization-iso}.
	
	Finally, we specialize the general classification to two basic examples: the Radford algebra \(R=\mathfrak A^3_{C_p}(g,f)\) and cyclic group algebras \(\K[C_n]\). We first classify all matched pairs between $R$ and $\K[C_n]$; they are parametrized by a scalar $\alpha\in\K$, with $\alpha=0$ forced when $p\nmid n$, while every scalar occurs when $p\mid n$. We also determine the unrestricted Hopf-isomorphism classes of the resulting bicrossed products $E_\alpha$; when $p\mid n$, these classes are the affine $\mathbb F_p$-orbits of $\alpha$.
	
	Using the exact factorization structure theorem, we then study the nontrivial exact factorizations of $E_\alpha$. When $n$ is prime, the list is explicit: if $n\ne p$, then $E_0$ has exactly one nontrivial exact factorization, namely
	\[
	E_0=R\cdot \K[C_n];
	\]
	if $n=p$, then every $E_\alpha$ has exactly $p$ nontrivial exact factorizations, given by
	\[
	E_\alpha=R\cdot \K[\langle hg^c\rangle],
	\qquad c\in\mathbb F_p.
	\]

	The paper is organized as follows. Section~\ref{sec:prelim} fixes conventions and recalls matched pairs, bicrossed products, Majid's theorem, the normalized third-type family, skew primitive spaces, Hopf subalgebras, and the form of Jacobson's formula used later. Section~\ref{sec:external} proves the rigidity results, classifies all matched pairs with a group algebra, and gives the isomorphism classification. Section~\ref{sec:internal} proves the exact factorization structure theorem and identifies the canonical matched pair attached to a group factorization. Section~\ref{sec:application} treats the Radford algebra and cyclic group algebras.

	\section{Preliminaries}\label{sec:prelim}
	
	\subsection*{Conventions}
	
	Throughout, $\K$ is an algebraically closed field of characteristic $p>0$. All vector spaces, algebras, coalgebras, Hopf algebras, and tensor products are over $\K$. All groups appearing below are finite unless  stated otherwise. We use Sweedler notation
	\[
	\Delta(u)=u_{(1)}\otimes u_{(2)}
	\]
	with the summation suppressed. For a Hopf algebra $C$, we denote by $\G(C)$ the set of group-like elements of $C$. For group-like elements $a,b$ of a Hopf algebra $B$, put
	\[
	\Pp_{a,b}(B)=\{z\in B\mid \Delta(z)=z\otimes a+b\otimes z\}.
	\]
	For a finite group $\Gamma$, let $\Gamma^{\mathrm{ab}}=\Gamma/[\Gamma,\Gamma]$ and set
	\[
	r_p(\Gamma):=\dim_{\mathbb F_p}\bigl(\Gamma^{\mathrm{ab}}\otimes_{\mathbb Z}\mathbb F_p\bigr).
	\]
	
	\subsection{Matched pairs, bicrossed products, and exact factorizations}
	The general theory of matched pairs and bicrossed products can be found in \cite{Takeuchi81,Majid90,Majid95}. A \emph{matched pair of Hopf algebras} is a quadruple $(A,H,\triangleright,\triangleleft)$
	consisting of a left action $\triangleright:H\otimes A\to A$ and a right action $\triangleleft:H\otimes A\to H$ such that $A$ is a left $H$-module coalgebra, $H$ is a right $A$-module coalgebra, and, for all $a,b\in A$ and $u,v\in H$,
	\begin{align}
		u\triangleright1_A&=\varepsilon_H(u)1_A,
		&1_H\triangleleft a&=\varepsilon_A(a)1_H,
		\tag{MP1}\label{eq:MP1}\\
		u\triangleright(ab)
		&=\bigl(u_{(1)}\triangleright a_{(1)}\bigr)
		\bigl((u_{(2)}\triangleleft a_{(2)})\triangleright b\bigr),
		\tag{MP2}\label{eq:MP2}\\
		(uv)\triangleleft a
		&=\bigl(u\triangleleft(v_{(1)}\triangleright a_{(1)})\bigr)
		\bigl(v_{(2)}\triangleleft a_{(2)}\bigr),
		\tag{MP3}\label{eq:MP3}\\
		u_{(1)}\triangleleft a_{(1)}\otimes u_{(2)}\triangleright a_{(2)}
		&=u_{(2)}\triangleleft a_{(2)}\otimes u_{(1)}\triangleright a_{(1)}.
		\tag{MP4}\label{eq:MP4}
	\end{align}
	The corresponding \emph{bicrossed product} $A\bowtie H$ is the tensor coalgebra $A\otimes H$, written $a\bowtie u$ on elementary tensors, with multiplication
	\begin{equation}\label{eq:bicross-product-mult}
		(a\bowtie u)(b\bowtie v)
		=a\bigl(u_{(1)}\triangleright b_{(1)}\bigr)
		\bowtie\bigl(u_{(2)}\triangleleft b_{(2)}\bigr)v.
	\end{equation}

	\begin{defi}[Exact factorization]\label{def:factor}
		A Hopf algebra $E$ is said to \emph{factor exactly through} Hopf algebras $A$ and $H$ if there are injective Hopf algebra maps $i:A\to E$ and $j:H\to E$ such that
		\[
		A\otimes H\longrightarrow E,
		\qquad a\otimes u\longmapsto i(a)j(u),
		\]
		is bijective. When $A$ and $H$ are fixed Hopf subalgebras of $E$, we simply write $E=AH$.
	\end{defi}
	
	\begin{thm} \cite[Proposition~3.12]{Majid90}\cite[Theorem~7.2.3]{Majid95}\label{thm:majid}
		Let $A$ and $H$ be Hopf algebras. A Hopf algebra $E$ factors exactly through $A$ and $H$ if and only if there exists a matched pair $(A,H,\triangleright,\triangleleft)$ such that
		\[
		E\cong A\bowtie H
		\]
		as Hopf algebras.
	\end{thm}

	The following lemma follows directly from the module-coalgebra axioms in the definition of a matched pair; see, e.g., \cite[Lemma 1.3]{Agore18-Taft}
	\begin{lem}\label{lem:matched-pair-skew}
		Let $(A,H,\triangleright,\triangleleft)$ be a matched pair of Hopf algebras, $a \in \G(A)$ and $u\in \G(H)$. Then:
		\begin{enumerate}
			\item $u \triangleright a \in \G(A)$ and $u \triangleleft a \in \G(H)$;
			\item If $x \in \Pp_{1,a}(A)$, then $u \triangleleft x \in \Pp_{u,u \triangleleft a}(H)$ and $u \triangleright x \in \Pp_{1,u \triangleright a}(A)$;
			\item If $y \in \Pp_{1,u}(H)$, then $y \triangleleft a \in \Pp_{1,u \triangleleft a}(H)$ and $y \triangleright a \in \Pp_{a,u \triangleright a}(A)$.
		\end{enumerate}
	\end{lem} 
	\subsection{The third-type rank-one pointed Hopf algebras}

	\begin{defi}[The third family $\mathfrak A^3$]\label{def:fA3}
		A \emph{third-type datum} is a triple $(G,g,f)$ such that
		\begin{enumerate}
			\item $G$ is a finite group and $g\in Z(G)$;
			\item $f:G\to(\K,+)$ is a group homomorphism;
			\item $f(g)=1$;
			\item for every $t\in G$,
			\begin{equation}\label{eq:A3-data}
				\bigl(f(t)^p-f(t)\bigr)(1-g^p)=0
				\qquad\text{in }\K[G].
			\end{equation}
		\end{enumerate}
		The Hopf algebra $\mathfrak A^3_G(g,f)$ is generated by the group algebra $\K[G]$ and an element $x$, subject to
		\begin{equation}\label{eq:A3-relations}
			tx-xt=f(t)t(1-g),
			\qquad x^p=x,
		\end{equation}
		where $t\in G$, with  coalgebra structure and antipode given by
		\begin{align}
			\Delta(t)&=t\otimes t,& \varepsilon(t)&=1,& S(t)&=t^{-1},\label{eq:A3-Hopf-group}\\
			\Delta(x)&=x\otimes1+g\otimes x,& \varepsilon(x)&=0,& S(x)&=-g^{-1}x.\label{eq:A3-Hopf-x}
		\end{align}
		It has PBW basis
		\begin{equation}\label{eq:A3-PBW}
			\{t x^i\mid t\in G,\ 0\le i<p\},
		\end{equation}
		and hence
		\[
		\dim\mathfrak A^3_G(g,f)=p|G|.
		\]
	\end{defi}
	
	\begin{rmk}\label{rmk:A3-data}
		Because $f(g)=1$ and $g$ has finite order, $p\mid\ord(g)$; in particular $g\neq1$. Condition \eqref{eq:A3-data} is automatic when $g^p=1$; otherwise, it is equivalent to $f(G)\subseteq\mathbb F_p.$ The relation with the distinguished group-like is
		\begin{equation}\label{eq:g-x-relation}
			gx-xg=g(1-g).
		\end{equation}
	\end{rmk}
	
	\begin{rmk}[Radford algebra]\label{rmk:Radford-fA3}
		Take $G=C_p=\langle g\rangle$ and $f(g)=1$. Then $f(g^i)=i\in\mathbb F_p$, $g^p=1$, and $\mathfrak A^3_{C_p}(g,f)$ is the Radford algebra. In the classification of pointed Hopf algebras of dimension $p^2$, this is the unique noncommutative and noncocommutative type \cite{WW14}.
	\end{rmk}

	\begin{lem}\label{lem:primitive-group}
		Let $\Gamma$ be a finite group and $A=\K[\Gamma]$. Then
		\[
		\Pp_{1,h}(A)=
		\begin{cases}
			0,&h=1,\\
			\K(1-h),&h\neq1.
		\end{cases}
		\]
	\end{lem}
	
	\begin{proof}
		This follows immediately by comparing coefficients in the basis $\{r\otimes s\mid r,s\in\Gamma\}$.
	\end{proof}
	
	\begin{lem}\label{lem:primitive-H}
		Let $H=\mathfrak A^3_G(g,f)$. Then
		\[
		\Pp_{1,h}(H)=
		\begin{cases}
			0,&h=1,\\[2mm]
			\K(1-g)\oplus\K x,&h=g,\\[2mm]
			\K(1-h),&h\neq1,g.
		\end{cases}
		\]
		In particular, $g$ is the unique group-like element $h\neq1$ for which $\Pp_{1,h}(H)$ contains an element outside the coradical $\K[G]$.
	\end{lem}

	\begin{lem}\label{lem:subalgebra}
		Let $H=\mathfrak A^3_G(g,f)$ and let $C\subseteq H$ be a Hopf subalgebra. Put
		\[
		G_C=C\cap G.
		\]
		Then exactly one of the following occurs:
		\begin{enumerate}
			\item $C\subseteq\K[G]$, in which case $C=\K[G_C]$;
			\item $C\not\subseteq\K[G]$, in which case $g\in G_C$, $x\in C$, and
			\[
			C=\mathfrak A^3_{G_C}\bigl(g,f|_{G_C}\bigr).
			\]
		\end{enumerate}
	\end{lem}
	\begin{proof}
		Since $H$ is pointed, so is every Hopf subalgebra $C$, and
		\[
		C_0=\K[G_C],\qquad G_C=C\cap G.
		\]
		If $C\subseteq\K[G]$, then it is immediate that $C=\K[G_C]$. Assume from now on that $C$ is not contained in $\K[G]$.
		
		We first claim that $C_1\ne C_0$. Consider the coradical filtration
		\[
		H_n=\sum_{0\le i\le\min\{n,p-1\}}\K[G]x^i,
		\]
		and the associated graded algebra $\operatorname{gr}H$. Its homogeneous basis is $\{t\bar x^i\mid t\in G,\ 0\le i<p\}$, where $\deg\bar x=1$. In $\operatorname{gr}H$, the element $\bar x$ commutes with $G$, satisfies $\bar x^p=0$, and
		\[
		\Delta(\bar x)=\bar x\otimes1+g\otimes\bar x.
		\]
		For the subcoalgebra $C$, we have $C_n=C\cap H_n$, so $\operatorname{gr}C$ is a graded Hopf subalgebra of $\operatorname{gr}H$. Let $n>0$ be the least integer such that $(\operatorname{gr}C)_n\ne0$, and choose
		\[
		0\ne y=\sum_{t\in G}c_t\,t\bar x^n\in(\operatorname{gr}C)_n.
		\]
		Then $1\le n<p$. If $n>1$, the bidegree $(n-1,1)$ component of $\Delta(y)$ is
		\[
		n\sum_{t\in G}c_t\,tg\bar x^{n-1}\otimes t\bar x.
		\]
		This is nonzero because $n\ne0$ in $\K$. On the other hand, since
		\[
		\Delta(y)\in\operatorname{gr}C\otimes\operatorname{gr}C,
		\]
		that component must lie in
		\[
		(\operatorname{gr}C)_{n-1}\otimes(\operatorname{gr}C)_1,
		\]
		which is zero by the minimality of $n$, a contradiction. Hence $n=1$, and therefore $C_1\ne C_0$.
		
		By the Taft--Wilson Theorem \cite{TaftWilson}, there exist $c\in G_C$ and
		\[
		z\in\Pp_{1,c}(C)\setminus\K[G_C].
		\]
		Since $\Pp_{1,c}(C)\subseteq\Pp_{1,c}(H)$, Lemma~\ref{lem:primitive-H} forces $c=g$, and then there exist $a\in\K^\times$ and $b\in\K$ such that
		\[
		z=ax+b(1-g).
		\]
		Thus $g\in G_C$ and $x\in C$.
		
		It remains to show that no additional group-like elements from outside $G_C$ occur. Since $x\in C$, the graded algebra $\operatorname{gr}C$ contains $\bar x$, and its degree-zero part is $\K[G_C]$. Let $0<n<p$, and take
		\[
		y=\sum_{t\in G}c_t\,t\bar x^n\in(\operatorname{gr}C)_n.
		\]
		The component of $\Delta(y)$ whose second tensor factor has degree zero is
		\[
		\sum_{t\in G}c_t\,t\bar x^n\otimes t.
		\]
		This belongs to $(\operatorname{gr}C)_n\otimes\K[G_C]$. By linear independence of group-like elements, every $t$ with $c_t\ne0$ must lie in $G_C$. Hence
		\[
		(\operatorname{gr}C)_n\subseteq\K[G_C]\bar x^n.
		\]
		The reverse inclusion is clear, since both $\K[G_C]$ and $\bar x$ lie in $\operatorname{gr}C$. Therefore
		\[
		\operatorname{gr}C=\K[G_C]\langle\bar x\rangle,
		\]
		and consequently $\dim C=p|G_C|$. But the Hopf subalgebra generated by $G_C$ and $x$ inside $H$ is precisely $\mathfrak A^3_{G_C}(g,f|_{G_C})$ and has the same dimension. Hence $C=\mathfrak A^3_{G_C}(g,f|_{G_C}).$
	\end{proof}

	\subsection{Jacobson's formula and a consequence}
	
	For an associative algebra $R$ and $y\in R$, write
	\[
	b(\operatorname{ad}_R y):=[b,y]=by-yb.
	\]
	
	The following results play an important role in the classification of pointed Hopf algebras in positive characteristic; see \cite{WW14} and \cite{Scherotzke}.
	\begin{pro}\label{proJ}
		Let $\K$ be a field of characteristic $p>0$, $R$ an associative $\K$-algebra, and $B\subseteq R$ a commutative subalgebra such that
		\[
		[B,y]\subseteq B
		\]
		for some $y\in R$. Then, for every $b\in B$,
		\begin{equation}\label{eq:Jacobson-special}
			(b+y)^p=b^p+y^p+b(\operatorname{ad}_R y)^{p-1}.
		\end{equation}
		For a given $\lambda\in\K$,	if $q\in B$ satisfies
		\[
		[q,y]=\lambda q(1-q),
		\]
		then
		\begin{equation}\label{eq:Jacobson-iterate}
			q(\operatorname{ad}_R y)^{p-1}
			=\lambda^{p-1}(q-q^p).
		\end{equation}
	\end{pro}
	\begin{proof}
		Formula \eqref{eq:Jacobson-special} is the standard special case of Jacobson's formula in which the iterated commutators remain in a commutative subalgebra; see \cite[pp.~186--187]{J}.
		
		For \eqref{eq:Jacobson-iterate}, set $D:=\operatorname{ad}_R y$ for short, 	so that
		\[
		D(q)=[q,y]=\lambda q(1-q).
		\]
		If $\lambda=0$, then $D(q)=0$, and hence
		\[
		D^{p-1}(q)=0=\lambda^{p-1}(q-q^p).
		\]
		Suppose therefore that $\lambda\neq0$. Put $\widetilde D=\lambda^{-1}D.$ Then $\widetilde D(q)=q(1-q).$ We claim that
		\[
		\widetilde D^{p-1}(q)=q-q^p.
		\]
		Since $\widetilde D=[\,\cdot\,,\lambda^{-1}y]$ is a derivation and
		$\widetilde D(q)=q(1-q)\in\K[q]$, it preserves $\K[q]$. Thus, for every
		polynomial $P(t)\in\K[t]$,
		\[
		\widetilde D(P(q))
		=P'(q)\widetilde D(q)
		=\bigl(t(1-t)P'(t)\bigr)(q).
		\]
		Consequently, it suffices to prove the polynomial identity
		\[
		\bigl(t(1-t)\partial_t\bigr)^{p-1}t=t-t^p
		\]
		in $\K[t]$.
		
		Work in the formal power series ring $\K[[t]]$ and set
		\[
		u=\frac{t}{1-t}.
		\]
		Then
		\[
		t=\frac{u}{1+u},
		\qquad
		1-t=\frac{1}{1+u},
		\]
		and
		\[
		t(1-t)\partial_t=u\partial_u.
		\]
		Moreover,
		\[
		t=\frac{u}{1+u}
		=\sum_{m\ge1}(-1)^{m-1}u^m.
		\]
		Since
		\[
		(u\partial_u)^{p-1}u^m=m^{p-1}u^m,
		\]
		we obtain
		\[
		(u\partial_u)^{p-1}t
		=
		\sum_{m\ge1}(-1)^{m-1}m^{p-1}u^m.
		\]
		In characteristic $p$, Fermat's little theorem gives
		\[
		m^{p-1}=
		\begin{cases}
			1,&p\nmid m,\\
			0,&p\mid m.
		\end{cases}
		\]
		Therefore
		\[
		(u\partial_u)^{p-1}t
		=
		\sum_{p\nmid m}(-1)^{m-1}u^m
		=
		\frac{u}{1+u}-\frac{u^p}{1+u^p}.
		\]
		By the Frobenius identity in characteristic $p$,
		\[
		t^p
		=\left(\frac{u}{1+u}\right)^p
		=\frac{u^p}{(1+u)^p}
		=\frac{u^p}{1+u^p},
		\]
		it follows that
		\[
		(u\partial_u)^{p-1}t=t-t^p.
		\]
		This proves the desired polynomial identity, and hence
		\[
		\widetilde D^{p-1}(q)=q-q^p.
		\]
		Multiplying both sides by $\lambda^{p-1}$ gives
		\[
		D^{p-1}(q)=\lambda^{p-1}(q-q^p),
		\]
		which is \eqref{eq:Jacobson-iterate}.
	\end{proof}

	\begin{cor}\label{cor:p-shift}
		In $H=\mathfrak A^3_G(g,f)$, for every $a\in\K$,
		\begin{equation}\label{eq:p-shift}
			\bigl(x+a(1-g)\bigr)^p-\bigl(x+a(1-g)\bigr)
			=(x^p-x)+(a^p-a)(1-g^p).
		\end{equation}
	\end{cor}
	
	\begin{proof}
		Apply Proposition~\ref{proJ} in the commutative subalgebra $\K[g]$ to $b=-ag$ and $y=x$. Since $[g,x]=g(1-g)$, equations \eqref{eq:Jacobson-special}--\eqref{eq:Jacobson-iterate} give
		\[
		(x-ag)^p=x^p-a^p g^p-a(g-g^p).
		\]
		Because the scalar $a$ is central,
		\[
		x+a(1-g)=(x-ag)+a,
		\]
		and taking the $p$th power and simplifying yields \eqref{eq:p-shift}.
	\end{proof}
	
	\section{Classification of matched pairs with group algebras}\label{sec:external}
	
	Throughout this section, we fix
	\[
	A=\K[\Gamma],\qquad H=\fA^3_G(g,f),
	\]
	where $\Gamma$ is a finite group and $(G,g,f)$ is a
	third-type datum. We classify all matched pairs
	$(A,H,\triangleright,\triangleleft)$.

	\subsection{Complete classification of matched pairs}
	\begin{thm}\label{thm:distinguished-rigidity}
		For every $h\in\Gamma$,
		\[
		g\triangleright h=h,
		\qquad
		x\triangleright h=0.
		\]
		Consequently, if $G=\langle g\rangle$, then the whole left action is trivial: for $u\in H,\ h\in\Gamma$,
		\[
		u\triangleright h=\varepsilon_H(u)h.
		\]
		
	\end{thm}
	\begin{proof}
		Fix $h\in\Gamma$. The right action $\triangleleft$ makes $H$ a right $A$-module coalgebra, so the map
		\[
		\rho_h:H\to H,\qquad \rho_h(u)=u\triangleleft h,
		\]
		is a coalgebra automorphism; in particular $\rho_h(H_0)=H_0$.
		
		Let $q=g\triangleright h$. By Lemma~\ref{lem:matched-pair-skew}(1), $q\in \G(A)=\Gamma$, and by Lemma~\ref{lem:matched-pair-skew}(3),
		\[
		x\triangleright h\in \Pp_{h,q}(A).
		\]
		Left multiplication by $h^{-1}$ identifies $\Pp_{h,q}(A)$ with $\Pp_{1,h^{-1}q}(A)$, so Lemma~\ref{lem:primitive-group} gives
		\[
		\Pp_{h,q}(A)=
		\begin{cases}
			0,&q=h,\\
			\K(h-q),&q\ne h.
		\end{cases}
		\]
		
		Assume $q\ne h$. Then $x\triangleright h=\lambda(h-q)$ for some $\lambda\in\K$. Applying \eqref{eq:MP4} to $(x,h)$, we obtain
		\[
		(x\triangleleft h)\otimes(h-q)
		=(1-g\triangleleft h)\otimes(x\triangleright h)
		=\lambda(1-g\triangleleft h)\otimes(h-q).
		\]
		Since $h-q\neq0$, cancellation in the tensor product gives
		\[
		x\triangleleft h=\lambda(1-g\triangleleft h)\in\K[G]=H_0.
		\]
		But $\rho_h$ preserves the coradical, so $x\notin H_0$ forces $x\triangleleft h\notin H_0$, a contradiction. Hence $q=h$. Consequently $g\triangleright h=h$, and then $x\triangleright h\in \Pp_{h,h}(A)=0$, as desired.
		
		Finally, if $G=\langle g\rangle$, then $H$ is generated as an algebra by $g$ and $x$. Since $g\triangleright h=h$ and $x\triangleright h=0$ for every $h\in\Gamma$, the $H$-module structure yields $u\triangleright h=\varepsilon_H(u)h$ for $u\in H,\ h\in\Gamma$.
	\end{proof}

	\begin{rmk}\label{rmk:nontrivial-left-action}
		The condition $G=\langle g\rangle$ in the last assertion of Theorem~\ref{thm:distinguished-rigidity} cannot be omitted. For instance, take $p>2$, $G=C_p\times C_2=\langle g\rangle\times\langle s\rangle$, and $f(g^is^j)=i\in\mathbb F_p$. Let $\Gamma=C_3=\langle h\rangle$. Define $\triangleright$ by
		\[
		g\triangleright h^i=h^i,\qquad s\triangleright h^i=h^{-i},\qquad x\triangleright h^i=0
		\]
		for all $i$, extended multiplicatively to an action of $H$ on $A$, and take $\triangleleft$ to be trivial. A direct verification shows that  these data satisfy the matched-pair axioms. Thus $s\triangleright h=h^{-1}\ne h$, so the left action need not be trivial for general $G$.
	\end{rmk}

	\begin{lem}\label{lem:g-fix}
		For every $h\in\Gamma$,
		\[
		g\triangleleft h=g,\qquad 	x\triangleleft h\in \Pp_{1,g}(H).
		\]
	\end{lem}
	\begin{proof}
		Since $x\in \Pp_{1,g}(H)$ and $h\in \G(A)$, Lemma~\ref{lem:matched-pair-skew}(3) gives
		\[
		x\triangleleft h\in \Pp_{1,g\triangleleft h}(H).
		\]
		The map $u\mapsto u\triangleleft h$ is a coalgebra automorphism, and $x\notin H_0$, hence $x\triangleleft h\notin H_0$. By Lemma~\ref{lem:primitive-H}, the only group-like $r$ for which $\Pp_{1,r}(H)$ contains a noncoradical element is $r=g$. Therefore $g\triangleleft h=g$ and then $x\triangleleft h\in \Pp_{1,g}(H)$.
	\end{proof}

	\begin{lem}\label{lem:x-and-phi}
		The restrictions of $\triangleright$ and $\triangleleft$ to $G\times\Gamma$ form a matched pair of groups.  Moreover, for every $h\in\Gamma$ there is a unique scalar $\gamma_h\in\K$ such that
		\begin{equation}\label{eq:x-right-general}
			x\triangleleft h=x+\gamma_h(1-g).
		\end{equation}
		
	\end{lem}
	\begin{proof}
		Let $t\in G$ and $h\in\Gamma$. By Lemma~\ref{lem:matched-pair-skew}(1), both $t\triangleright h\in\Gamma$ and $t\triangleleft h\in G$. Thus $\triangleright$ and $\triangleleft$ restrict to maps
		\[
		G\times\Gamma\to\Gamma,\qquad G\times\Gamma\to G.
		\]
		Restricting \eqref{eq:MP2} and \eqref{eq:MP3} to group-like elements gives the identities
		\[
		t\triangleright(hk)=(t\triangleright h)\bigl((t\triangleleft h)\triangleright k\bigr),
		\]
		\[
		(ts)\triangleleft h=\bigl(t\triangleleft(s\triangleright h)\bigr)(s\triangleleft h),
		\]
		which are precisely the matched-pair axioms for the groups $\Gamma$ and $G$.
		
		Now fix $h\in\Gamma$. Since $x\in \Pp_{1,g}(H)$ and $h\in \G(A)$, Lemma~\ref{lem:matched-pair-skew}(3) gives
		\[
		x\triangleleft h\in \Pp_{1,g\triangleleft h}(H)=\Pp_{1,g}(H),
		\]
		because $g\triangleleft h=g$ by Lemma~\ref{lem:g-fix}. Moreover, the right action of $h$ is a coalgebra automorphism and $x\notin H_0$, so $x\triangleleft h\notin H_0$. Hence Lemma~\ref{lem:primitive-H} yields
		\[
		x\triangleleft h=\beta_h x+\gamma_h(1-g)
		\]
		with $\beta_h\in\K^\times$ and $\gamma_h\in\K$.
		
		Apply $\triangleleft h$ to the relation $gx-xg=g(1-g)$. Using \eqref{eq:MP3} together with $g\triangleright h=h$, $x\triangleright h=0$, and $g\triangleleft h=g$, we obtain
		\[
		(gx)\triangleleft h=g(x\triangleleft h),\qquad
		(xg)\triangleleft h=(x\triangleleft h)g,
		\]
		and
		\[
		(g(1-g))\triangleleft h=g(1-g).
		\]
		Substituting $x\triangleleft h=\beta_h x+\gamma_h(1-g)$ into the left-hand side gives
		\[
		\beta_h(gx-xg)=g(1-g).
		\]
		Since $gx-xg=g(1-g)\neq0$, it follows that $\beta_h=1$. The scalar $\gamma_h$ is uniquely determined because $1-g\neq0$. Hence
		\[
		x\triangleleft h=x+\gamma_h(1-g).
		\]
	\end{proof}

	\begin{lem}\label{lem:constraint}
		Let $\gamma_h$ be defined by \eqref{eq:x-right-general}.  Then
		\begin{enumerate}
			\item $\gamma:\Gamma\to(\K,+)$ is a group homomorphism;
			\item $\gamma_h\in \K_g$ for every $h$, where
			\begin{equation}\label{eq:Kg-def}
				\K_g:=\{a\in\K\mid(a^p-a)(1-g^p)=0\}
				=\begin{cases}
					\K,&g^p=1,\\
					\mathbb F_p,&g^p\ne1;
				\end{cases}
			\end{equation}
			\item for all $t\in G$ and $h\in\Gamma$,
			\begin{equation}\label{eq:compat-f-gamma}
				f(t\triangleleft h)
				=f(t)+\gamma_{t\triangleright h}-\gamma_h.
			\end{equation}
		\end{enumerate}
	\end{lem}
	\begin{proof}
		(1) For $h,k\in\Gamma$, the right module law gives
		\[
		(x\triangleleft h)\triangleleft k=x\triangleleft(hk).
		\]
		Using $g\triangleleft k=g$ and \eqref{eq:x-right-general}, the left-hand side is
		\[
		x+(\gamma_h+\gamma_k)(1-g),
		\]
		while the right-hand side is $x+\gamma_{hk}(1-g)$. Since $1-g\neq0$, it follows that $\gamma_{hk}=\gamma_h+\gamma_k$.
		
		(2) Since $x^p=x$, applying $\triangleleft h$ to both sides and using
		\[
		(x^m)\triangleleft h=(x\triangleleft h)^m\qquad(m\ge1),
		\]
		which follows by induction from \eqref{eq:MP3}, we get $(x\triangleleft h)^p=x\triangleleft h$. Thus
		\[
		(x+\gamma_h(1-g))^p=x+\gamma_h(1-g).
		\]
		By Corollary~\ref{cor:p-shift} with $a=\gamma_h$,
		\[
		(\gamma_h^p-\gamma_h)(1-g^p)=0,
		\]
		so $\gamma_h\in \K_g$. If $g^p\neq1$, then $1-g^p\neq0$, and because $\gamma_h^p-\gamma_h$ is a scalar, we must have $\gamma_h^p=\gamma_h$. Hence $\gamma_h\in\mathbb F_p$, which proves \eqref{eq:Kg-def}.
		
		(3) Fix $t\in G$ and $h\in\Gamma$. Apply $\triangleleft h$ to
		\[
		tx-xt=f(t)t(1-g).
		\]
		Using \eqref{eq:MP3}, Theorem~\ref{thm:distinguished-rigidity}, and Lemma~\ref{lem:g-fix}, we obtain
		\begin{align*}
			(tx)\triangleleft h&=(t\triangleleft h)(x\triangleleft h),\\
			(xt)\triangleleft h&=(x\triangleleft(t\triangleright h))(t\triangleleft h),\\
			(t(1-g))\triangleleft h&=(t\triangleleft h)(1-g).
		\end{align*}
		Substituting \eqref{eq:x-right-general} and using the defining relation for $t\triangleleft h\in G$, we get
		\[
		\bigl(f(t\triangleleft h)+\gamma_h-\gamma_{t\triangleright h}\bigr)(t\triangleleft h)(1-g)
		=f(t)(t\triangleleft h)(1-g).
		\]
		Since $(t\triangleleft h)(1-g)\neq0$, the desired compatibility \eqref{eq:compat-f-gamma} follows.
	\end{proof}

	\begin{thm}\label{thm:right-action}
		There is a bijection between matched pairs $(A,H,\triangleright,\triangleleft)$ and data
		\[
		(\triangleright_\Gamma,\triangleleft_G,\gamma)
		\]
		with the following properties:
		\begin{enumerate}
			\item $(\Gamma,G,\triangleright_\Gamma,\triangleleft_G)$ is a matched pair of finite groups;
			\item $g\triangleright_\Gamma h=h$ and $g\triangleleft_G h=g$ for every $h\in\Gamma$;
			\item $\gamma\in\Hom(\Gamma,\K_g^+)$;
			\item
			\[
			f(t\triangleleft_G h)
			=f(t)+\gamma_{t\triangleright_\Gamma h}-\gamma_h
			\qquad(t\in G,\ h\in\Gamma).
			\]
		\end{enumerate}
		For the corresponding Hopf matched pair, the actions on the additional generator are
		\[
		x\triangleright h=0,
		\qquad
		x\triangleleft h=x+\gamma_h(1-g).
		\]
		If $L=\Gamma\bowtie G$ is the group bicrossed product and
		\begin{equation}\label{eq:F-on-L}
			F:L\to\K,
			\qquad
			F(h,t)=f(t)-\gamma_h,
		\end{equation}
		then
		\begin{equation}\label{eq:bicross-is-A3}
			A\bowtie H\cong\fA^3_L(g,F)
		\end{equation}
		as Hopf algebras.
	\end{thm}
	
	\begin{proof}
		Starting from a Hopf matched pair, Theorem~\ref{thm:distinguished-rigidity} and Lemmas~\ref{lem:g-fix}--\ref{lem:constraint} produce exactly the data listed in the statement. Conversely, assume that such data are given.
		
		Let $L=\Gamma\bowtie G$ be the group bicrossed product with multiplication
		\[
		(h,t)(k,s)=\bigl(h(t\triangleright_\Gamma k),(t\triangleleft_G k)s\bigr).
		\]
		Condition (2) and $g\in Z(G)$ imply that the element $g=(1,g)$ is central in $L$. Define $F:L\to\K$ by $F(h,t)=f(t)-\gamma_h$. Then $F$ is a group homomorphism: for $(h,t),(k,s)\in L$,
		\begin{align*}
			F\bigl((h,t)(k,s)\bigr)
			&=f((t\triangleleft_G k)s)-\gamma_{h(t\triangleright_\Gamma k)}\\
			&=f(t\triangleleft_G k)+f(s)-\gamma_h-\gamma_{t\triangleright_\Gamma k}\\
			&=f(t)+f(s)-\gamma_h-\gamma_k\\
			&=F(h,t)+F(k,s),
		\end{align*}
		where the third equality uses condition (4). Also $F(g)=f(g)-\gamma_1=1$. If $g^p=1$, the condition $(F(\ell)^p-F(\ell))(1-g^p)=0$ is automatic. If $g^p\ne1$, then by Definition~\ref{def:fA3} we have $f(G)\subseteq\mathbb F_p$, and by condition (3) we have $\gamma(\Gamma)\subseteq\mathbb F_p$; hence $F(L)\subseteq\mathbb F_p$, so the same condition holds. Thus $\fA^3_L(g,F)$ is well-defined.
		
		Inside $\fA^3_L(g,F)$, the subgroup $\Gamma$ generates a Hopf subalgebra isomorphic to $A$, while $G$ together with $x$ generates a Hopf subalgebra isomorphic to $H$. Indeed, the defining relations for $H$ are reproduced because $F(1,t)=f(t)$, and the relations involving $h\in\Gamma$ are
		\[
		hx-xh=-\gamma_h h(1-g),
		\]
		or equivalently
		\[
		xh=h\bigl(x+\gamma_h(1-g)\bigr).
		\]
		Using the PBW basis $\{\ell x^i\mid \ell\in L,\ 0\le i<p\}$ and the unique normal form $ht$ for elements of $L$, multiplication gives a vector-space isomorphism
		\[
		A\otimes H\longrightarrow \fA^3_L(g,F).
		\]
		Thus $\fA^3_L(g,F)$ is an exact factorization through $A$ and $H$. By Majid's theorem, this factorization corresponds to a unique matched pair $(A,H,\triangleright,\triangleleft)$. The cross-relations
		\[
		th=(t\triangleright_\Gamma h)(t\triangleleft_G h),\qquad
		xh=h\bigl(x+\gamma_h(1-g)\bigr)
		\]
		show that this matched pair has exactly the prescribed actions. This proves the converse and the isomorphism \eqref{eq:bicross-is-A3}.
	\end{proof}

	\begin{cor}\label{cor:parameter-space}
		Assume that the left action is trivial.  Put
		\[
		\Aut_{g,f}(G)
		=\{\sigma\in\Aut(G)\mid \sigma(g)=g,\ f\circ\sigma=f\}.
		\]
		Then matched pairs with trivial left action are in bijection with
		\begin{equation}\label{eq:TL-parameter}
			\Hom\bigl(\Gamma,\Aut_{g,f}(G)^{\mathrm{op}}\bigr)
			\times\Hom(\Gamma,\K_g^+).
		\end{equation}
		Explicitly, if $(\phi,\gamma)$ is such a pair, then
		\[
		t\triangleleft h=\phi_h(t),
		\qquad
		x\triangleleft h=x+\gamma_h(1-g).
		\]
		If $r=r_p(\Gamma)$, then
		\[
		\Hom(\Gamma,\K_g^+)\cong
		\begin{cases}
			\K^r,&g^p=1,\\
			\mathbb F_p^r,&g^p\ne1.
		\end{cases}
		\]
	\end{cor}
	\begin{proof}
		If the left action is trivial, then Lemma~\ref{lem:x-and-phi} restricts to a matched pair of groups with $t\triangleright h=h$. For each $h\in\Gamma$, the map
		\[
		\phi_h:G\to G,\qquad \phi_h(t)=t\triangleleft h
		\]
		is therefore a group automorphism, and the assignment $h\mapsto\phi_h$ defines a homomorphism $\Gamma\to\Aut(G)^{\mathrm{op}}$. Lemma~\ref{lem:g-fix} gives $\phi_h(g)=g$, while the compatibility \eqref{eq:compat-f-gamma} reduces to $f\circ\phi_h=f$. Hence $\phi_h\in\Aut_{g,f}(G)$. Together with $\gamma\in\Hom(\Gamma,\K_g^+)$ from Theorem~\ref{thm:right-action}, this proves the asserted bijection.
		
		For the last assertion, recall that $\K_g^+$ is an abelian group of exponent $p$. Therefore every homomorphism from $\Gamma$ to $\K_g^+$ factors through
		\[
		\Gamma^{\mathrm{ab}}/p\Gamma^{\mathrm{ab}}\cong(\mathbb F_p)^r.
		\]
		If $g^p=1$, then $\K_g=\K$, and
		\[
		\Hom(\Gamma,\K^+)\cong\Hom_{\mathbb F_p}((\mathbb F_p)^r,\K)\cong\K^r.
		\]
		If $g^p\ne1$, then $\K_g=\mathbb F_p$, and
		\[
		\Hom(\Gamma,\mathbb F_p^+)\cong(\mathbb F_p)^r.
		\]
	\end{proof}

	\subsection{Isomorphism classification of bicrossed products}
	\begin{lem}\label{lem:aut-H}
		Every Hopf automorphism $\Phi$ of $H=\fA^3_G(g,f)$ has the form
		\begin{equation}\label{eq:aut-H-form}
			\Phi(t)=\sigma(t)\quad(t\in G),
			\qquad
			\Phi(x)=x+a(1-g),
		\end{equation}
		where $\sigma\in\Aut_{g,f}(G)$ and $a\in \K_g$.  Conversely, every pair $(\sigma,a)$ with these properties defines a Hopf automorphism.  Hence
		\[
		\Aut_{\mathrm{Hopf}}(H)\cong\Aut_{g,f}(G)\times \K_g^+.
		\]
	\end{lem}

	\begin{proof}
		A Hopf automorphism preserves the coradical $\K[G]$, hence induces some $\sigma\in\Aut(G)$. By Lemma~\ref{lem:primitive-H}, $g$ is characterized among the nontrivial group-like elements by the existence of a noncoradical element in $\Pp_{1,g}(H)$. Since $x$ is such an element, we get $\sigma(g)=g$, and
		\[
		\Phi(x)=\beta x+a(1-g)
		\]
		for some $\beta\in\K^\times$ and $a\in\K$.
		
		Applying $\Phi$ to $gx-xg=g(1-g)$ and using $\sigma(g)=g$ gives
		\[
		\beta(gx-xg)=g(1-g).
		\]
		Since $gx-xg=g(1-g)\neq0$, we obtain $\beta=1$.
		
		For $t\in G$, applying $\Phi$ to $tx-xt=f(t)t(1-g)$ yields
		\[
		\sigma(t)x-x\sigma(t)=f(t)\sigma(t)(1-g),
		\]
		because $1-g$ commutes with every group-like element. Comparing with the defining relation for $\sigma(t)$ gives
		\[
		f(\sigma(t))=f(t),
		\]
		so $\sigma\in\Aut_{g,f}(G)$.
		
		Finally, applying $\Phi$ to $x^p=x$ and using \eqref{eq:p-shift} yields
		\[
		(a^p-a)(1-g^p)=0,
		\]
		hence $a\in\K_g$.
		
		Conversely, the same calculations show that \eqref{eq:aut-H-form} preserves all defining relations and the coalgebra structure. Composition of such automorphisms corresponds to composing the $\sigma$'s and adding the scalars, which gives the direct-product decomposition.
	\end{proof}

	\begin{thm}[Isomorphism classification]\label{thm:isomorphism}
		Let $D$ and $D'$ be two matched-pair data from Theorem~\ref{thm:right-action}, and let
		\[
		(L,F),\qquad(L',F')
		\]
		be the associated group/function pairs from \eqref{eq:F-on-L}.  Write $E_D=A\bowtie H$ and $E_{D'}=A\bowtie H$ for the corresponding bicrossed products.  Then
		\[
		E_D\cong E_{D'}
		\quad\text{as Hopf algebras}
		\]
		if and only if there exists a group isomorphism
		\[
		\Theta:L\longrightarrow L'
		\]
		such that
		\begin{equation}\label{eq:full-iso-condition}
			\Theta(g)=g,
			\qquad
			F'\circ\Theta=F.
		\end{equation}
		
		If, in addition, both matched pairs have trivial left action and one restricts to \emph{factor-preserving} Hopf isomorphisms, then for parameters $(\phi,\gamma)$ and $(\phi',\gamma')$ the criterion is equivalently the existence of
		$\sigma\in\Aut_{g,f}(G)$ and $\psi\in\Aut(\Gamma)$ such that
		\begin{equation}\label{eq:factor-preserving-iso}
			\phi'_h
			=\sigma\circ\phi_{\psi^{-1}(h)}\circ\sigma^{-1},
			\qquad
			\gamma'_h=\gamma_{\psi^{-1}(h)}
			\quad(h\in\Gamma).
		\end{equation}
	\end{thm}

	\begin{proof}
		By Theorem~\ref{thm:right-action},
		$E_D\cong\fA^3_L(g,F)$ and
		$E_{D'}\cong\fA^3_{L'}(g,F')$. Let
		$\Phi:\fA^3_L(g,F)\to\fA^3_{L'}(g,F')$ be a Hopf isomorphism. Its restriction to group-like elements is a group isomorphism
		$\Theta:L\to L'$. Since the distinguished skew primitive is characterized by the relation with $g$, we get $\Theta(g)=g$ and
		\[
		\Phi(x)=x+a(1-g)
		\]
		for some $a\in\K_g$, as in Lemma~\ref{lem:aut-H}.
		
		For $\ell\in L$, applying $\Phi$ to
		\[
		\ell x-x\ell=F(\ell)\ell(1-g)
		\]
		and using that $a(1-g)$ commutes with every group-like element, we obtain
		\[
		\Theta(\ell)x-x\Theta(\ell)
		=F(\ell)\Theta(\ell)(1-g).
		\]
		Comparing with the defining relation for $\Theta(\ell)$ in the target gives
		$F'(\Theta(\ell))=F(\ell)$, proving \eqref{eq:full-iso-condition}. Conversely, a group isomorphism satisfying \eqref{eq:full-iso-condition}, together with $x\mapsto x$, preserves all defining relations and hence extends to a Hopf isomorphism.
		
		Now suppose both left actions are trivial and $\Phi$ preserves the two factors. Then $\Theta$ restricts to automorphisms
		$\psi\in\Aut(\Gamma)$ and $\sigma\in\Aut_{g,f}(G)$. Preservation of the group bicrossed-product multiplication gives
		\[
		\sigma\circ\phi_k
		=\phi'_{\psi(k)}\circ\sigma,
		\]
		which is the first identity in \eqref{eq:factor-preserving-iso}. The condition
		$F'\circ\Theta=F$ yields
		\[
		f(\sigma(t))-\gamma'_{\psi(h)}=f(t)-\gamma_h.
		\]
		Since $f\circ\sigma=f$, this reduces to $\gamma'_{\psi(h)}=\gamma_h$, which is equivalent to
		\[
		\gamma'_h=\gamma_{\psi^{-1}(h)}.
		\]
		The converse is immediate by defining $\Theta(h,t)=(\psi(h),\sigma(t))$.
		
		In general, the group isomorphism $\Theta$ need not preserve the two group-like factors; such mixing occurs explicitly in Section~\ref{sec:application}.
	\end{proof}
	
	\subsection{Special cases}\label{sec:special}
	The following corollaries specialize Theorem~\ref{thm:right-action} to several cases of interest.
	
	\begin{cor}\label{cor:p-coprime}
		Assume $p\nmid|\Gamma|$.  Then $\gamma=0$ in every matched pair.  Hence the general matched pairs are exactly the matched pairs of groups
		$(\Gamma,G,\triangleright,\triangleleft)$ satisfying
		\[
		g\triangleright h=h,
		\qquad g\triangleleft h=g,
		\qquad f(t\triangleleft h)=f(t)
		\]
		for all $t\in G$ and $h\in\Gamma$.
		If the left action is trivial, this reduces to
		\[
		\Hom\bigl(\Gamma,\Aut_{g,f}(G)^{\mathrm{op}}\bigr).
		\]
	\end{cor}
	
	\begin{proof}
		The additive group $\K_g^+$ has exponent $p$, so the image of
		$\gamma:\Gamma\to \K_g^+$ is a finite $p$-group whose order divides $|\Gamma|$. Since $p\nmid|\Gamma|$, the image is trivial. The remaining assertions follow from Theorem~\ref{thm:right-action} and Corollary~\ref{cor:parameter-space}.
	\end{proof}
	
	\begin{cor}\label{cor:p-group}
		Assume that $\Gamma$ is a finite $p$-group and that the left action is trivial.  Let
		$r=r_p(\Gamma)$.  Then the matched-pair parameter space is
		\[
		\Hom\bigl(\Gamma,\Aut_{g,f}(G)^{\mathrm{op}}\bigr)
		\times
		\begin{cases}
			\K^r,&g^p=1,\\
			\mathbb F_p^r,&g^p\ne1.
		\end{cases}
		\]
		More explicitly:
		\begin{enumerate}
			\item if $\Gamma=C_{p^m}$, the $\phi$-component is determined by one
			$\sigma\in\Aut_{g,f}(G)$ satisfying $\sigma^{p^m}=\id$, and $\gamma$ is determined by one element of $\K_g$;
			\item if $\Gamma=(C_p)^r$, the $\phi$-component is a commuting $r$-tuple
			$(\sigma_1,\ldots,\sigma_r)$ in $\Aut_{g,f}(G)$ with
			$\sigma_i^p=\id$, and the $\gamma$-component is an arbitrary element of $\K_g^r$.
		\end{enumerate}
	\end{cor}
	\begin{proof}
		By Corollary~\ref{cor:parameter-space}, the parameter space is
		\[
		\Hom\bigl(\Gamma,\Aut_{g,f}(G)^{\mathrm{op}}\bigr)
		\times\Hom(\Gamma,\K_g^+),
		\]
		and the second factor is isomorphic to $\K^r$ or $\mathbb F_p^r$ according to whether $g^p=1$ or not.
		
		For the first factor, a homomorphism $\Gamma\to\Aut_{g,f}(G)^{\mathrm{op}}$ is determined by the images of a chosen set of generators. If $\Gamma=C_{p^m}=\langle h\rangle$, this is a single element $\sigma=\phi_h\in\Aut_{g,f}(G)$ with $\sigma^{p^m}=\id$. If $\Gamma=(C_p)^r$, the images of the $r$ generators must pairwise commute in $\Aut_{g,f}(G)$ and each has order dividing $p$, giving a commuting $r$-tuple $(\sigma_1,\ldots,\sigma_r)$ with $\sigma_i^p=\id$. The description of $\gamma$ follows from additivity.
	\end{proof}

	\begin{cor}\label{cor:G-Cp}
		Let $G=C_p=\langle g\rangle$, $f(g)=1$, and $H=\fA^3_{C_p}(g,f)$. Then for every finite group $\Gamma$, in any matched pair $(\K[\Gamma],H,\triangleright,\triangleleft)$ the left action is trivial and the right action on $G$ is trivial. Matched pairs are therefore in bijection with
		\[
		\Hom(\Gamma,\K^+).
		\]
	\end{cor}
	
	\begin{proof}
		Since $G=\langle g\rangle$, Theorem~\ref{thm:distinguished-rigidity} gives trivial left action. Moreover, $\Aut_{g,f}(C_p)=\{\id\}$ because any such automorphism fixes the generator $g$. Finally, $g^p=1$, so $\K_g=\K$. By Corollary~\ref{cor:parameter-space}, the parameter space is
		\[
		\Hom\bigl(\Gamma,\Aut_{g,f}(G)^{\mathrm{op}}\bigr)\times\Hom(\Gamma,\K_g^+),
		\]
		which reduces to $\Hom(\Gamma,\K^+)$.
	\end{proof}

	\section{Exact factorizations of $\mathfrak{A}^3$}\label{sec:internal}
	
	In this section we classify exact factorizations of $H=\fA^3_G(g,f).$

	\subsection{The factorization structure theorem}
	
	\begin{thm}[Factorization structure]\label{thm:factorization}
		Up to interchanging the two factors, exact factorizations
		\[
		H=BT
		\]
		by proper Hopf subalgebras are in bijection with nontrivial exact group
		factorizations
		\begin{equation}\label{eq:group-exact-factorization}
			G=G_B\Gamma,\qquad
			G_B\cap\Gamma=\{1\},\qquad
			g\in G_B,\qquad
			\Gamma\ne\{1\}.
		\end{equation}
		Under this correspondence,
		\[
		B=\fA^3_{G_B}\bigl(g,f|_{G_B}\bigr),
		\qquad
		T=\K[\Gamma].
		\]
	\end{thm}
	\begin{proof}
		Suppose  that $H=BT$ is an exact factorization. Then $B\cap T=\K1$.
		By Lemma~\ref{lem:subalgebra}, every Hopf subalgebra of $H$ is either a group algebra contained in $\K[G]$, or is of the form $\fA^3_{G_C}\bigl(g,f|_{G_C}\bigr)$ for some subgroup $G_C\le G$ with $g\in G_C$. In the latter case the subalgebra contains both $g$ and the skew-primitive generator $x$. Therefore two noncoradical Hopf subalgebras would both contain $x$, contradicting $B\cap T=\K1$. Thus at most one of $B,T$ lies outside $\K[G]$. On the other hand, they cannot both be contained in $\K[G]$, because then $BT\subseteq\K[G]\ne H$. Hence, after interchanging the factors if necessary, we may write
		\[
		B=\fA^3_{G_B}\bigl(g,f|_{G_B}\bigr),\qquad T=\K[\Gamma]
		\]
		for subgroups $G_B,\Gamma\le G$ with $g\in G_B$.
		
		From  $B\cap T=\K1$, we get  $G_B\cap\Gamma=\{1\}$. Since
		\[
		\dim H
		=p|G|
		=\dim B\cdot\dim T
		=p|G_B|\cdot|\Gamma|,
		\]
		we have $|G|=|G_B||\Gamma|$. Since $G_B\cap\Gamma=\{1\}$, the product $G_B\Gamma$ has order $|G_B||\Gamma|$, so $G=G_B\Gamma$. Finally, $\Gamma\ne\{1\}$ because $T$ is proper.
		
		Conversely, assume that subgroups $G_B,\Gamma\le G$ satisfy \eqref{eq:group-exact-factorization}, and put
		\[
		B=\fA^3_{G_B}\bigl(g,f|_{G_B}\bigr),\qquad T=\K[\Gamma].
		\]
		Both are Hopf subalgebras of $H$. It remains to show that multiplication $B\otimes T\to H$ is bijective.
		
		The PBW basis of $B\otimes T$ is
		\[
		\{\,s x^i\otimes h\mid s\in G_B,\ h\in\Gamma,\ 0\le i<p\,\}.
		\]
		For $h\in\Gamma$, the defining relation in $H$ can be written as
		\[
		xh=hx-f(h)h(1-g).
		\]
		Since $g\in G_B$, the element $1-g$ lies in $B$. By induction on $i$, this relation gives
		\[
		sx^ih=(sh)x^i+\text{a linear combination of terms of $x$-degree $<i$}.
		\]
		The exact group factorization $G=G_B\Gamma$ implies that
		\[
		(s,h)\mapsto sh
		\]
		is a bijection $G_B\times\Gamma\to G$. Hence, with the PBW basis $\{t x^i\mid t\in G,\ 0\le i<p\}$ of $H$, the multiplication map is represented, after ordering by $x$-degree, by a block-triangular matrix whose diagonal blocks are permutation matrices. In particular, the matrix is invertible, so multiplication is bijective. Therefore $H=BT$ is an exact factorization.
	\end{proof}

	\begin{rmk}\label{rmk:rigidity}
		The proof of Theorem~\ref{thm:factorization} relies on the fact that every noncoradical Hopf subalgebra of $H$ contains both $g$ and $x$. Consequently, an exact factorization by proper Hopf subalgebras cannot have two noncoradical factors: exactly one factor is a group algebra, while the other is a third-type rank-one Hopf algebra with the same distinguished pair $(g,x)$.
	\end{rmk}
	
	\begin{cor}
		\label{cor:Radford-no-factorization}
		The Radford algebra $R$ admits no exact factorization by two proper Hopf subalgebras.
	\end{cor}
	
	\begin{proof}
		By Theorem~\ref{thm:factorization}, an exact factorization of $R$ would require a nontrivial exact group factorization
		\[
		C_p=G_B\Gamma
		\]
		with $g\in G_B$. Since $g$ generates $C_p$, necessarily $G_B=C_p$, forcing $\Gamma=\{1\}$, contrary to nontriviality.
	\end{proof}
	
	\subsection{The canonical matched pair attached to a factorization}
	
	Let $G=G_B\Gamma$ be an exact group factorization as in
	\eqref{eq:group-exact-factorization}. Since the two subgroups are finite,
	also $G=\Gamma G_B$, and every product $th$, with
	$t\in G_B$ and $h\in\Gamma$, has a unique decomposition
	\begin{equation}\label{eq:canonical-group-actions}
		th=(t\triangleright h)(t\triangleleft h),
		\qquad
		t\triangleright h\in\Gamma,\quad
		t\triangleleft h\in G_B.
	\end{equation}
	These maps form the standard matched pair of groups associated with the
	exact group factorization.

	\begin{thm}[Canonical matched pair]\label{thm:factorization-mp}
		Let $G=G_B\Gamma$ satisfy \eqref{eq:group-exact-factorization}, and set
		\[
		A=\K[\Gamma],\qquad B=\fA^3_{G_B}\bigl(g,f|_{G_B}\bigr).
		\]
		The matched pair of Hopf algebras associated by Majid's theorem with the reversed exact factorization
		\[
		H=\K[\Gamma]\,\fA^3_{G_B}\bigl(g,f|_{G_B}\bigr)
		\]
		is the one determined in Theorem~\ref{thm:right-action} by the group actions in \eqref{eq:canonical-group-actions} and
		\begin{equation}\label{eq:canonical-gamma}
			\gamma_h=-f(h)\qquad(h\in\Gamma).
		\end{equation}
		Equivalently,
		\[
		x\triangleright h=0,\qquad x\triangleleft h=x-f(h)(1-g).
		\]
		For this matched pair, the group bicrossed product $L=\Gamma\bowtie G_B$ is canonically isomorphic to $G$ via
		\[
		\Theta:L\longrightarrow G,\qquad \Theta(h,t)=ht,
		\]
		and the additive map $F$ of Theorem~\ref{thm:right-action} satisfies
		\[
		F(h,t)=f(ht)=f\bigl(\Theta(h,t)\bigr).
		\]
		Consequently,
		\[
		A\bowtie B\cong\fA^3_L(g,F)\cong\fA^3_G(g,f)=H.
		\]
	\end{thm}
	
	\begin{proof}
		The group actions are precisely those arising from the exact group factorization, so $\Theta(h,t)=ht$ is the usual group isomorphism $\Gamma\bowtie G_B\to G$.
		
		It remains to determine the action on $x$. In $H$, for $h\in\Gamma$,
		\[
		hx-xh=f(h)h(1-g),
		\]
		hence
		\[
		xh=h\bigl(x-f(h)(1-g)\bigr).
		\]
		On the other hand, in the bicrossed product \eqref{eq:bicross-product-mult}, the cross relation is computed as
		\[
		xh=(x_{(1)}\triangleright h)(x_{(2)}\triangleleft h)
		=(x\triangleright h)\cdot 1+(g\triangleright h)(x\triangleleft h)
		=h(x\triangleleft h),
		\]
		because $x\triangleright h=0$ and $g\triangleright h=h$. Comparing with the previous expression yields
		\[
		x\triangleleft h=x-f(h)(1-g),
		\]
		which is \eqref{eq:canonical-gamma}.
		
		Finally, Theorem~\ref{thm:right-action} gives
		\[
		F(h,t)=f(t)-\gamma_h=f(t)+f(h)=f(ht)=f\bigl(\Theta(h,t)\bigr),
		\]
		since $f$ is a group homomorphism and $ht$ is the product in $G$. The asserted Hopf-algebra isomorphisms follow.
	\end{proof}

	\begin{cor}[Exact factorizations up to Hopf automorphism]
		\label{cor:factorization-iso}
		After orienting each factorization so that the rank-one factor is written first, two exact factorizations
		\[
		H=\fA^3_{G_{B_1}}\bigl(g,f|_{G_{B_1}}\bigr)\K[\Gamma_1]
		\quad\text{and}\quad
		H=\fA^3_{G_{B_2}}\bigl(g,f|_{G_{B_2}}\bigr)\K[\Gamma_2]
		\]
		are equivalent under $\Aut_{\mathrm{Hopf}}(H)$ if and only if there exists
		\[
		\sigma\in\Aut_{g,f}(G)
		\]
		such that
		\[
		\sigma(G_{B_1})=G_{B_2},
		\qquad
		\sigma(\Gamma_1)=\Gamma_2.
		\]
		Thus equivalence classes of exact factorizations are precisely the $\Aut_{g,f}(G)$-orbits on the group exact factorizations \eqref{eq:group-exact-factorization}.
	\end{cor}
	
	\begin{proof}
		By Lemma~\ref{lem:aut-H}, every Hopf automorphism of $H$ has the form
		\[
		t\longmapsto\sigma(t)\quad(t\in G),
		\qquad
		x\longmapsto x+a(1-g),
		\]
		with $\sigma\in\Aut_{g,f}(G)$ and $a\in\K_g$. Since $\sigma$ acts on group-like elements as a group automorphism of $G$, it sends the subgroup $G_{B_1}$ to $\sigma(G_{B_1})$ and $\Gamma_1$ to $\sigma(\Gamma_1)$. Moreover, because $g\in G_{B_i}$, the translation term $a(1-g)$ lies in each rank-one factor $\fA^3_{G_{B_i}}(g,f|_{G_{B_i}})$, so it does not change that factor as a subalgebra. This proves necessity.
		
		Conversely, if such $\sigma$ exists, Lemma~\ref{lem:aut-H} with $a=0$ provides a Hopf automorphism of $H$ carrying the first rank-one factor to the second and $\K[\Gamma_1]$ to $\K[\Gamma_2]$. Hence the factorizations are equivalent.
	\end{proof}

	\section{The Radford algebra and cyclic groups}\label{sec:application}
	
	Let
	\[
	R=\fA^3_{C_p}(g,f),
	\qquad C_p=\langle g\rangle,
	\qquad f(g)=1,
	\]
	and let $\Gamma=C_n=\langle h\rangle$. Thus
	\[
	g^p=1,\qquad x^p=x,
	\qquad gx-xg=g(1-g),
	\qquad
	\Delta(x)=x\otimes1+g\otimes x.
	\]
	
	\begin{lem}\label{lem:Aut-Radford}
		For the Radford datum,
		\[
		\Aut_{g,f}(C_p)=\{\id\}.
		\]
		Nevertheless, the full Hopf automorphism group of $R$ is nontrivial:
		\[
		\Aut_{\mathrm{Hopf}}(R)\cong(\K,+),
		\qquad x\longmapsto x+a(1-g).
		\]
	\end{lem}
	
	\begin{proof}
		Any $\sigma\in\Aut_{g,f}(C_p)$ must fix the distinguished generator $g$. Since $C_p$ is generated by $g$, this forces $\sigma=\id$.
		
		By Lemma~\ref{lem:aut-H}, every Hopf automorphism of $R$ has the form
		\[
		t\mapsto\sigma(t)\quad(t\in C_p),\qquad x\mapsto x+a(1-g),
		\]
		where $\sigma\in\Aut_{g,f}(C_p)$ and $a\in\K_g$. As shown above, $\sigma=\id$, and since $g^p=1$, we have $\K_g=\K$. Hence the Hopf automorphisms are precisely the translations
		\[
		x\mapsto x+a(1-g),\qquad a\in\K,
		\]
		which gives $\Aut_{\mathrm{Hopf}}(R)\cong(\K,+)$.
	\end{proof}
	\begin{thm}[Matched pairs for $R$ and $C_n$]\label{thm:Radford-cyclic}
		Every matched pair $(\K[C_n],R,\triangleright,\triangleleft)$ is determined by a scalar $\alpha\in\K$ through
		\[
		g\triangleleft h=g,
		\qquad
		x\triangleleft h=x+\alpha(1-g),
		\]
		with the following alternatives:
		\begin{enumerate}
			\item if $p\nmid n$, necessarily $\alpha=0$;
			\item if $p\mid n$, every $\alpha\in\K$ occurs.
		\end{enumerate}
		The corresponding bicrossed product $E_\alpha$ is generated by $g,h,x$ with relations
		\begin{align}
			g^p&=1,& h^n&=1,& gh&=hg,\label{eq:Radford-group-rel}\\
			x^p&=x,& gx-xg&=g(1-g),& xh-hx&=\alpha h(1-g),\label{eq:Radford-x-rel}
		\end{align}
		and coalgebra structure
		\[
		\Delta(g)=g\otimes g,
		\qquad
		\Delta(h)=h\otimes h,
		\qquad
		\Delta(x)=x\otimes1+g\otimes x.
		\]
		Equivalently,
		\[
		E_\alpha\cong\fA^3_{C_n\times C_p}(g,F_\alpha),
		\qquad
		F_\alpha(h^ig^j)=j-i\alpha.
		\]
	\end{thm}

	\begin{proof}
		By Corollary~\ref{cor:G-Cp}, matched pairs correspond to homomorphisms $\gamma:C_n\to\K^+$. Put $\alpha=\gamma_h$. Since $h^n=1$, we must have
		\[
		0=\gamma_{h^n}=n\alpha.
		\]
		If $p\nmid n$, then $n\cdot1_\K$ is nonzero in $\K$, hence invertible, so $\alpha=0$. If $p\mid n$, then $n\cdot1_\K=0$, so $n\alpha=0$ for every $\alpha\in\K$.
		
		Since the left action is trivial and the right action on $G$ is trivial, the group generators commute, and the only nontrivial cross relation is
		\[
		xh=h(x\triangleleft h)=hx+\alpha h(1-g).
		\]
		The presentation follows. Finally, by Theorem~\ref{thm:right-action},
		\[
		F_\alpha(h^ig^j)=f(g^j)-\gamma_{h^i}=j-i\alpha,
		\]
		because $f(g^j)=j$ and $\gamma_{h^i}=i\alpha$. This gives the stated isomorphism.
	\end{proof}

	\begin{cor}\label{cor:Radford-trivial}
		If $p\nmid n$, then
		\[
		E_0\cong\K[C_n]\otimes R.
		\]
	\end{cor}
	\begin{proof}
		In this case Theorem~\ref{thm:Radford-cyclic} gives $\alpha=0$. Hence both the left action and the right action are trivial, so all generators of the two tensor factors commute. Therefore the bicrossed product reduces to the ordinary tensor product.
	\end{proof}
	
	\begin{thm}[Unrestricted Hopf isomorphisms]\label{thm:Radford-iso}
		Assume $p\mid n$.  For $\alpha,\alpha'\in\K$,
		\[
		E_\alpha\cong E_{\alpha'}
		\quad\Longleftrightarrow\quad
		\alpha'=u\alpha+v
		\quad\text{for some }u\in\mathbb F_p^\times,
		\ v\in\mathbb F_p.
		\]
		Equivalently, the set of Hopf-isomorphism classes is the orbit space
		\[
		\K/\operatorname{AGL}_1(\mathbb F_p),
		\]
		where $\operatorname{AGL}_1(\mathbb F_p)$ acts by affine transformations
		$a\mapsto ua+v$.
	\end{thm}

	\begin{proof}
		By Theorem~\ref{thm:Radford-cyclic},
		$E_\alpha\cong\fA^3_L(g,F_\alpha)$ with $L=C_n\times C_p$. Theorem~\ref{thm:isomorphism} states that an isomorphism $E_\alpha\to E_{\alpha'}$ is equivalent to a group automorphism
		$\Theta:L\to L$ fixing $g$ and satisfying
		\[
		F_{\alpha'}\circ\Theta=F_\alpha.
		\]
		
		Every automorphism of $L$ fixing $g$ has the form
		\begin{equation}\label{eq:theta-h-radford}
			\Theta(h)=g^c h^\zeta,
		\end{equation}
		for some $c\in\mathbb F_p$ and $\zeta\in(\mathbb Z/n\mathbb Z)^\times$. Indeed, the induced automorphism of
		$L/\langle g\rangle\cong C_n$ sends the class of $h$ to its $\zeta$-th power, and the remaining factor lies in $\langle g\rangle$. Since $p\mid n$, every $c$ is compatible with the relation $h^n=1$.
		
		Let $\bar\zeta\in\mathbb F_p^\times$ be the reduction of $\zeta$ modulo $p$. Evaluating
		$F_{\alpha'}\circ\Theta=F_\alpha$ on $h$ gives
		\[
		c-\bar\zeta\alpha'=-\alpha,
		\]
		hence
		\begin{equation}\label{eq:alpha-iso-zeta}
			\alpha'=\bar\zeta^{-1}(\alpha+c).
		\end{equation}
		Conversely, whenever \eqref{eq:alpha-iso-zeta} holds, the map
		$g\mapsto g$, $h\mapsto g^ch^\zeta$, $x\mapsto x$ extends to a Hopf isomorphism.
		
		The reduction map
		$(\mathbb Z/n\mathbb Z)^\times\to\mathbb F_p^\times$ is surjective when $p\mid n$: writing $n=p^mq$ with $(p,q)=1$, one may prescribe any nonzero residue modulo $p$ and simultaneously choose a residue prime to $q$ by the Chinese remainder theorem. Hence \eqref{eq:alpha-iso-zeta} is equivalent to
		\[
		\alpha'=u\alpha+v
		\]
		with $u\in\mathbb F_p^\times$ and $v\in\mathbb F_p$.
	\end{proof}
		\begin{rmk}\label{rmk:factor-preserving-Radford}
		If one restricts to factor-preserving isomorphisms, then $c=0$ in \eqref{eq:theta-h-radford}. The isomorphism criterion thus reduces to scaling by the image of $(\mathbb Z/n\mathbb Z)^\times$ in $\mathbb F_p^\times$.
	\end{rmk}

	\begin{cor}\label{cor:Radford-classes}
		Assume $p\mid n$.  All parameters in $\mathbb F_p$ belong to the same Hopf-isomorphism class as $0$.  More generally, two parameters are isomorphic precisely when they lie in the same affine $\mathbb F_p$-orbit.
	\end{cor}

	\begin{proof}
		By Theorem~\ref{thm:Radford-iso}, translations by any $v\in\mathbb F_p$ are allowed. Taking $u=1$ sends $0$ to every $v\in\mathbb F_p$, proving the first assertion. The second assertion is exactly the isomorphism criterion in Theorem~\ref{thm:Radford-iso}.
	\end{proof}

\begin{cor}[Automorphism groups of the bicrossed products]\label{cor:aut-Ealpha}
	Assume $p\mid n$, and let $E_\alpha$ be the bicrossed product from
	Theorem~\ref{thm:Radford-cyclic}.  Then
	\[
	\Aut_{\mathrm{Hopf}}(E_\alpha)
	\cong
	\K^+
	\times
	\begin{cases}
		(\mathbb Z/n\mathbb Z)^\times, & \alpha\in\mathbb F_p,\\[3mm]
		\Ker\bigl((\mathbb Z/n\mathbb Z)^\times
		\longrightarrow
		(\mathbb Z/p\mathbb Z)^\times\bigr),
		& \alpha\notin\mathbb F_p.
	\end{cases}
	\]
	The factor $\K^+$ corresponds to the translations
	$x\mapsto x+a(1-g)$, $a\in\K$.
\end{cor}

\begin{proof}
	By Theorem~\ref{thm:right-action},
	\[
	E_\alpha\cong \mathfrak A^3_{L}(g,F_\alpha),
	\qquad
	L=C_n\times C_p,
	\]
	where
	\[
	F_\alpha(h^i g^j)=j-i\alpha.
	\]
	By Lemma~\ref{lem:aut-H},
	\[
	\Aut_{\mathrm{Hopf}}\bigl(\mathfrak A^3_L(g,F_\alpha)\bigr)
	\cong
	\Aut_{g,F_\alpha}(L)\times \K^+,
	\]
	where
	\[
	\Aut_{g,F_\alpha}(L)
	=
	\{\sigma\in\Aut(L)\mid \sigma(g)=g,\ F_\alpha\circ\sigma=F_\alpha\}.
	\]
	Since $L=C_n\times C_p$ and $g$ generates the $C_p$-factor, every $\sigma\in\Aut(L)$ satisfying $\sigma(g)=g$ has the form
	\[
	\sigma(g)=g,\qquad
	\sigma(h)=h^\zeta g^c,
	\]
	for some $\zeta\in(\mathbb Z/n\mathbb Z)^\times$ and $ c\in\mathbb F_p$. Conversely, since $p\mid n$, we have $(g^c)^n=1$, and since
	$\zeta\in(\mathbb Z/n\mathbb Z)^\times$, every such pair $(\zeta,c)$ defines an automorphism of $L$ fixing $g$.

	Let $\sigma\in\Aut_{g,F_\alpha}(L)$.  Then $F_\alpha\circ\sigma=F_\alpha$. In particular, taking $i=1$ and $j=0$ gives $F_\alpha(\sigma(h))=F_\alpha(h)$. Now $\sigma(h)=h^\zeta g^c$, so
 
	\[
	F_\alpha(\sigma(h))=c-\zeta\alpha,
	\]
	while $F_\alpha(h)=-\alpha$.  Hence
	\begin{equation}\label{eq:aut_E}
	c=(\zeta-1)\alpha.
    \end{equation}
	Here $\zeta$ is reduced modulo $p$, so that $\zeta-1$ is viewed
	as an element of $\mathbb F_p$.
	
	Suppose first that $\alpha\in\mathbb F_p$.  Then
	$(\zeta-1)\alpha\in\mathbb F_p$ for every
	$\zeta\in(\mathbb Z/n\mathbb Z)^\times$.  Thus, for every such
	$\zeta$, equation~\eqref{eq:aut_E} determines a unique
	\[
	c=(\zeta-1)\alpha\in\mathbb F_p.
	\]
	Conversely, every element of $\Aut_{g,F_\alpha}(L)$ arises in this
	way.  Hence the projection
	\[
	\Aut_{g,F_\alpha}(L)\longrightarrow
	(\mathbb Z/n\mathbb Z)^\times,
	\qquad
	\sigma\longmapsto\zeta,
	\]
	is an isomorphism, and therefore
	\[
	\Aut_{g,F_\alpha}(L)
	\cong
	(\mathbb Z/n\mathbb Z)^\times.
	\]
	
	Now suppose that $\alpha\notin\mathbb F_p$.  Since
	$c\in\mathbb F_p$, equation~\eqref{eq:aut_E} implies
	\[
	(\zeta-1)\alpha\in\mathbb F_p.
	\]
	If $\zeta-1\neq0$ in $\mathbb F_p$, then $\zeta-1$ is invertible
	in $\mathbb F_p$, and consequently
	\[
	\alpha=\frac{c}{\zeta-1}\in\mathbb F_p,
	\]
	a contradiction.  Thus
	\[
	\zeta\equiv1\pmod p.
	\]
	Equation~\eqref{eq:aut_E} then gives $c=0$.
	
	Conversely, let
	\[
	\zeta\in(\mathbb Z/n\mathbb Z)^\times
	\quad\text{with}\quad
	\zeta\equiv1\pmod p,
	\]
	and define
	\[
	\sigma(g)=g,\qquad \sigma(h)=h^\zeta.
	\]
	Then for every $i,j$,
	\[
	F_\alpha(\sigma(h^i g^j))
	=
	F_\alpha(h^{\zeta i}g^j)
	=
	j-\zeta i\alpha.
	\]
	Since $\zeta\equiv1\pmod p$, we have
	$\zeta i\equiv i\pmod p$.  Hence in the characteristic-$p$ field $\K$,
	\[
	j-\zeta i\alpha
	=
	j-i\alpha.
	\]
	Thus $F_\alpha\circ\sigma=F_\alpha$, so
	\[
	\sigma\in\Aut_{g,F_\alpha}(L).
	\]
	Therefore
	\[
	\Aut_{g,F_\alpha}(L)
	\cong
	\{\zeta\in(\mathbb Z/n\mathbb Z)^\times
	\mid \zeta\equiv1\pmod p\}.
	\]
	This subgroup is precisely
	\[
	\Ker\bigl((\mathbb Z/n\mathbb Z)^\times
	\longrightarrow
	(\mathbb Z/p\mathbb Z)^\times\bigr).
	\]
	
	Combining the two cases with Lemma~\ref{lem:aut-H} gives
	\[
	\Aut_{\mathrm{Hopf}}(E_\alpha)
	\cong
	\K^+
	\times
	\begin{cases}
		(\mathbb Z/n\mathbb Z)^\times, & \alpha\in\mathbb F_p,\\[3mm]
		\Ker\bigl((\mathbb Z/n\mathbb Z)^\times
		\longrightarrow
		(\mathbb Z/p\mathbb Z)^\times\bigr),
		& \alpha\notin\mathbb F_p.
	\end{cases}
	\]
	The $\K^+$-factor is precisely the group of translations
	$x\mapsto x+a(1-g)$, $a\in\K$.
\end{proof}

	\begin{pro}[General exact factorizations of $E_\alpha$]\label{pro:general-factorizations}
		Let $E_\alpha$ be as in Theorem~\ref{thm:Radford-cyclic}, and write $n=p^m q$ with $m\ge0$ and $p\nmid q$. 
		The nontrivial exact factorizations of $E_\alpha$ are in bijection with choices of exact group factorizations
		\[
		C_{p^m}\times C_p=P_B\times P_\Gamma,\qquad
		C_q=Q_B\times Q_\Gamma,
		\]
		such that $g\in P_B$ and $P_\Gamma\times Q_\Gamma\neq\{1\}$. 
		For each such choice, the corresponding exact factorization is $E_\alpha=BT$, where
		\[
		B=\fA^3_{P_B\times Q_B}\bigl(g,F_\alpha|_{P_B\times Q_B}\bigr),
		\qquad
		T=\K[P_\Gamma\times Q_\Gamma].
		\]
	\end{pro}
	\begin{proof}
		By Theorems~\ref{thm:Radford-cyclic} and~\ref{thm:factorization}, exact factorizations of $E_\alpha$ by proper Hopf subalgebras correspond to nontrivial exact factorizations
		\[
		C_n\times C_p=G_B\,\Gamma,\qquad G_B\cap\Gamma=\{1\},\qquad g\in G_B.
		\]
		Since $C_n\times C_p$ is abelian, such a factorization is a direct product decomposition: $G_B$ and $\Gamma$ are complementary direct factors.
		
		Write
		\[
		C_n\times C_p\cong C_{p^m}\times C_p\times C_q,
		\]
		where $C_{p^m}\times C_p$ is the $p$-primary component and $C_q$ is the $p'$-primary component. Decomposing $G_B$ and $\Gamma$ according to primary components, we obtain
		\[
		G_B=P_B\times Q_B,\qquad \Gamma=P_\Gamma\times Q_\Gamma
		\]
		with
		\[
		P_B=G_B\cap(C_{p^m}\times C_p),\quad Q_B=G_B\cap C_q,
		\]
		and similarly for $P_\Gamma,Q_\Gamma$. Thus the group factorization is the direct product of
		\[
		C_{p^m}\times C_p=P_B\times P_\Gamma,\qquad C_q=Q_B\times Q_\Gamma.
		\]
		Since $g$ lies in the $p$-primary component, the condition $g\in G_B$ is equivalent to $g\in P_B$. Nontriviality is equivalent to $P_\Gamma\times Q_\Gamma\neq\{1\}$.
		
		For each such decomposition, Theorem~\ref{thm:factorization} yields the corresponding Hopf subalgebras
		\[
		B=\fA^3_{P_B\times Q_B}\bigl(g,F_\alpha|_{P_B\times Q_B}\bigr),
		\qquad
		T=\K[P_\Gamma\times Q_\Gamma].
		\]
		This establishes the stated bijection.
	\end{proof}
	
	\begin{cor}\label{cor:p-coprime-factorizations}
		If $p\nmid n$, then $\alpha=0$ by Theorem~\ref{thm:Radford-cyclic}. In this case the data reduce to $P_B=\langle g\rangle\cong C_p$, $P_\Gamma=1$, and exact factorizations
		\[
		C_n=Q_B\times Q_\Gamma,\qquad Q_\Gamma\neq1.
		\]
		These are parameterized by divisors $d>1$ of $n$ with $\gcd(d,n/d)=1$; here $Q_\Gamma$ and $Q_B$ are the unique subgroups of orders $d$ and $n/d$, respectively. The corresponding factorizations are
		\[
		E_0=\fA^3_{\langle g\rangle\times Q_B}\bigl(g,F_0|_{\langle g\rangle\times Q_B}\bigr)\cdot \K[Q_\Gamma].
		\]
	\end{cor}
	
	\begin{proof}
		When $p\nmid n$, the $p$-primary part is just $C_p$, so $g\in P_B$ forces $P_B=C_p$ and $P_\Gamma=1$. The rest follows from Proposition~\ref{pro:general-factorizations} and the standard description of direct factors of a cyclic group.
	\end{proof}
	
	\begin{cor}\label{cor:Radford-prime-factorizations}
		Let $n$ be a prime, and write $C_n=\langle h\rangle$ as before.
		\begin{enumerate}
			\item If $n\neq p$, then $\alpha=0$ and the unique nontrivial factorization of $E_0$ is
			\[
			E_0=R\cdot \K[C_n].
			\]
			\item If $n=p$, then for every $\alpha\in\K$, the nontrivial factorizations of $E_\alpha$ are
			\[
			E_\alpha=R\cdot \K[\langle hg^c\rangle],\qquad c\in\mathbb F_p,
			\]
			where in both cases $R=\fA^3_{C_p}(g,f)$ with $f(g^j)=j$.
		\end{enumerate}
	\end{cor}
	
	\begin{proof}
		If $n\neq p$, then $n$ is coprime to $p$, so Corollary~\ref{cor:p-coprime-factorizations} applies. Since $n$ is prime, its only divisor $d>1$ is $d=n$, giving $Q_\Gamma=C_n$ and $Q_B=1$. The unique nontrivial factorization is therefore
		\[
		E_0=R\cdot\K[C_n],
		\]
		where $R=\fA^3_{C_p}(g,f)$ with $f(g^j)=j$.
		
		Now suppose $n=p$. Then $m=1$ and $q=1$, so Proposition~\ref{pro:general-factorizations} reduces to exact factorizations
		\[
		C_p\times C_p=P_B\times P_\Gamma,\qquad g\in P_B,\qquad P_\Gamma\neq1.
		\]
		Both factors must have order $p$, so they are one-dimensional subspaces of $C_p\times C_p$. Since $g\in P_B$ and $g\neq1$, we get $P_B=\langle g\rangle$. The possible complements are exactly the lines different from $\langle g\rangle$, namely
		\[
		P_\Gamma=\langle hg^c\rangle,\qquad c\in\mathbb F_p,
		\]
		which gives $p$ choices.
		
		Finally, $F_\alpha(g^j)=j$ for all $j$, independent of $\alpha$, so $F_\alpha|_{\langle g\rangle}=f$. Hence $B=\fA^3_{\langle g\rangle}(g,f)=R$, and the corresponding factorizations are
		\[
		E_\alpha=R\cdot\K[\langle hg^c\rangle],\qquad c\in\mathbb F_p.
		\]
	\end{proof}

	\section*{Acknowledgments}
	
     This work was partially supported by the National Natural Science Foundation of China (Grant No.~12401041). The author thanks the authors of the works cited below for their inspiring contributions.

\end{document}